\documentclass[10pt]{article}

\usepackage{amsmath}
\usepackage{mathtools}
\usepackage{amssymb}
\usepackage{amsthm}
\usepackage{subcaption}
\usepackage{tikz}
\usepackage{tikz-cd}
\usepackage{hyperref}
\usepackage[capitalize,compress]{cleveref}
\usepackage[numbers,sort]{natbib}
\usepackage[shortlabels]{enumitem}

\allowdisplaybreaks[2]
\mathtoolsset{mathic=true}

\usetikzlibrary{calc, shapes.geometric}
\tikzset{face/.style={fill=black!15, fill opacity=0.7, semithick, line join=round}, vertex/.style={circle, inner sep=0.9pt, fill}, guide/.style={very thin}, every path/.style={semithick}}

\renewcommand{\emptyset}{\varnothing}

\renewcommand{\hat}{\widehat}
\newcommand{\Mu}{\mathrm{M}}
\newcommand{\Nu}{\mathrm{N}}

\DeclareMathOperator{\Skel}{Skel}

\DeclarePairedDelimiter{\abs}{\lvert}{\rvert}

\newcommand{\Squelta}{\mathord{{\tikz[baseline=0.5pt]{\fill [even odd rule] (0,0) [rounded corners=0.15pt] rectangle (0.68em,0.68em) (0.04em,0.07em) [sharp corners] rectangle (0.68em-0.09em,0.68em-0.04em); \useasboundingbox (-0.08em,0) rectangle (0.68em+0.08em,0.68em);}}}}

\theoremstyle{plain}
\newtheorem{thm}{Theorem}[section]
\newtheorem{lma}[thm]{Lemma}
\newtheorem{prop}[thm]{Proposition}
\newtheorem{crl}[thm]{Corollary}

\theoremstyle{definition}
\newtheorem*{dfn}{Definition}

\theoremstyle{remark}

\Crefname{thm}{Theorem}{Theorems}
\Crefname{lma}{Lemma}{Lemmas}
\Crefname{prop}{Proposition}{Propositions}
\Crefname{crl}{Corollary}{Corollaries}
\Crefname{conj}{Conjecture}{Conjectures}
\Crefname{dfn}{Definition}{Definitions}
\Crefname{claim}{Claim}{Claims}
\Crefname{rmk}{Remark}{Remarks}
\Crefname{eg}{Example}{Examples}
\Crefname{page}{Page}{Pages}

\title{Reconstructing $d$-manifold subcomplexes of cubes from their $(\lfloor d/2 \rfloor + 1)$-skeletons}
\author{Rowan Rowlands \\ {\small University of Washington, \url{rowanr@uw.edu}}\footnote{MSC classes: 52B70 (Primary), 05E45, 52B05, 57P10 (Secondary). Keywords: cubical complexes, manifolds, Poincar\'e duality, simplicial complexes}}

\begin{document}

\maketitle

\begin{abstract}
In 1984, Dancis proved that any $d$-dimensional simplicial manifold is determined by its $\left( \left\lfloor \frac{d}{2} \right\rfloor + 1 \right)$-skeleton. This paper adapts his proof to the setting of cubical complexes that can be embedded into a cube of arbitrary dimension. Under some additional conditions (for example, if the cubical manifold is a sphere), the result can be tightened to the $\left\lceil \frac{d}{2} \right\rceil$-skeleton when $d \geq 3$.
\end{abstract}

\section{Introduction}

\subsection{Polytope reconstruction}

The problem of reconstructing polytopes or related structures from partial information has seen much interest. To reconstruct an arbitrary $d$-dimensional polytope, for example, one needs to know its $(d-2)$-skeleton; however, some special classes of polytopes are determined by much less information: for example, zonotopes and simple $d$-dimensional polytopes are determined by their graphs, i.e.\ their $1$-skeletons. \Citet{art:Bayer} gives an excellent survey of this field.

Simplicial polytopes, and more generally simplicial spheres and manifolds, lie in between these two extremes. In \citep{art:Dancis}, Dancis presents a neat homological proof that $d$-dimensional simplicial manifolds are determined by their $\left( \left\lfloor \frac{d}{2} \right\rfloor + 1 \right)$-skeleton, generalising an idea for simplicial spheres attributed to Perles. In this paper, we modify Dancis's argument to prove a similar result for a class of cubical manifolds, namely those that can be embedded as cubical complexes into the standard cube $I^n$ for some (potentially large) $n$.

To reconstruct the $(k+1)$-skeleton of a manifold (simplicial or cubical) from its $k$-skeleton, we need to determine what the $(k+1)$-dimensional faces are. If $F$ is a $(k+1)$-dimensional face, its boundary is a subcomplex of the $k$-skeleton isomorphic to the boundary of a simplex or cube. Therefore, determining the $(k+1)$-dimensional faces of the complex amounts to answering the following question: if $S^k$ is a subcomplex of the $k$-skeleton that is isomorphic to the boundary of a $(k+1)$-dimensional simplex or cube, is it actually the boundary of a $(k+1)$-dimensional face of the complex, or did it just show up by coincidence? Dancis's key insight was that this question can be answered with knowledge of just the $k$-skeleton, if $k$ is large enough compared to $d$, thanks to Poincar\'e duality. In this paper, we adapt Dancis's argument to cubical manifolds that are embeddable in cubes. Our main results are \cref{thm:main,thm:main-tighter,thm:main-dimensions}.

\subsection{On subcomplexes of the cube}

At first glance, restricting our attention to complexes that can be embedded as a subcomplex of a cube may seem like a strict and arbitrary condition. However, we present three reasons why it is natural to consider these types of complexes.

First, there is the analogy to the simplicial world. Every simplicial complex is a subcomplex of a simplex: a simplicial complex is a collection of \emph{some} subsets of $\{1, \dotsc, n\}$, while the $(n-1)$-simplex is the collection of \emph{all} subsets of $\{1, \dotsc, n\}$.

Second, cubical complexes often arise in applied settings as subcomplexes of the standard grid of unit cubes in $\mathbb R^n$ (see \citep{book:Kaczynski-et-al}). \citet[Theorem~2.7]{art:Blass-Hosztynski} prove that any finite subcomplex of this grid can be embedded in some cube $I^N$ (with $N$ potentially much larger than $n$).

Third, many common families of cubical complexes can be embedded in the cube. This list includes boundaries of cubical zonotopes \citep[Corollary~7.17]{book:Ziegler}, stacked cubical polytopes (also called ``capped''), the neighbourly cubical spheres constructed by \citet{art:BBC}, $\mathrm{CAT}(0)$ complexes \citep[p.~149]{art:Ardila}, and cubical barycentric subdivisions of simplicial complexes \citep[Theorem~1.1]{art:Blass-Hosztynski}. Similarly, many operations on cubical complexes preserve the property of being a subcomplex of a cube, like product and disjoint union, cubical fissuring as defined by \citet{art:BBC}, and certain cubical versions of Pachner's bistellar moves considered by \citet[Proposition~3.7.3]{art:Funar}.

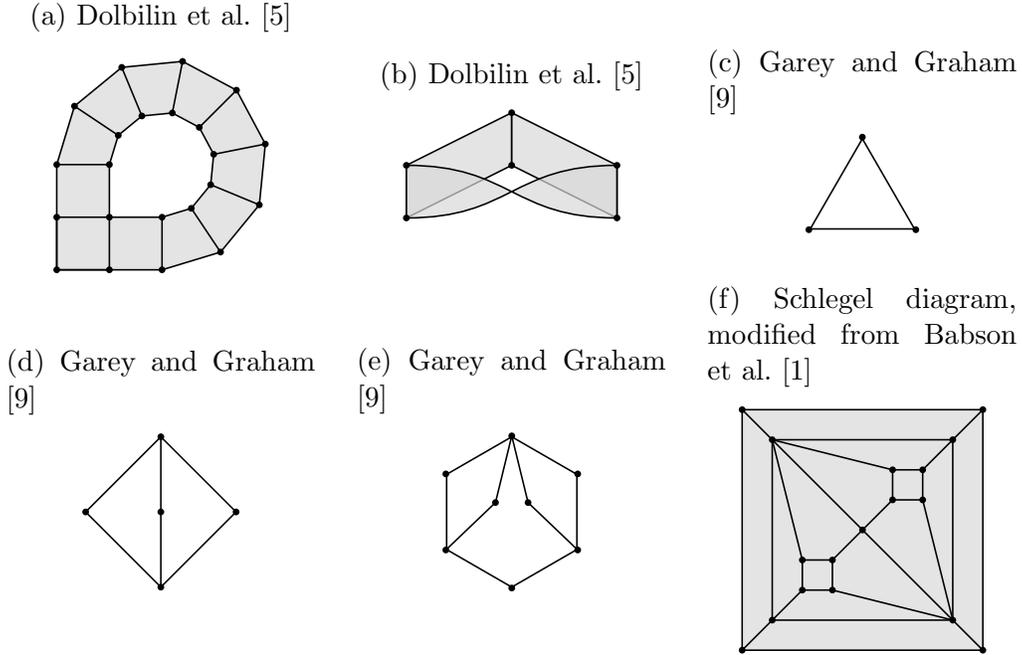
On the other hand, we are obliged to point out that there are many cubical complexes that cannot be embedded into any cube. Fortunately, though, there is a fairly straightforward criterion for determining whether a given cubical complex can be embedded in a cube, for which we will need the following under-appreciated lemma.

\begin{lma}[{\citet[Lemma~12]{art:Ehrenborg-Hetyei}}] \label{thm:Ehrenborg-Hetyei}
If $f : V(I^m) \to V(I^n)$ is an injection from the vertices of the cube $I^m$ to vertices of $I^n$ which maps edges to edges, that is, an embedding of the underlying graph $\mathcal G(I^m)$ into $\mathcal G(I^n)$, then the image of $f$ is a face of $I^n$.
\end{lma}

\begin{crl} \label{thm:cube-is-flag}
A cubical complex $\Squelta$ embeds into $I^n$ if and only if its underlying graph $\mathcal G(\Squelta)$ embeds into $\mathcal G(I^n)$.
\end{crl}

So the question of whether a complex can embed into a cube is a graph-theoretic question --- which is answered by the following theorem:

\begin{thm}[{\citet[Proposition~1]{art:Havel-Moravek}}]
A connected graph $\mathcal H$ can be embedded in $\mathcal G(I^n)$ if and only if there exists a labelling of the edges of $\mathcal H$ with labels $\{1, \dotsc, n\}$ such that:
\begin{itemize}
	\item in every path in $\mathcal H$, some label appears an odd number of times, and
	\item in every cycle in $\mathcal H$, no label appears an odd number of times.
\end{itemize}
\end{thm}

Note that this criterion is similar to a condition given by \citet[Theorem~2]{art:Dolbilin-et-al2} for a cubical manifold to be embeddable in the standard grid of unit cubes in $\mathbb R^n$.

\subsection{Acknowledgments}

This research was partially supported by a graduate fellowship from NSF grant DMS-1664865.

We thank Steven Klee, Margaret Bayer, Raman Sanyal, the anonymous reviewers, and especially Isabella Novik for many helpful suggestions.

\section{Definitions}

Cubical complexes are a variation on the well-studied notion of simplicial complexes, where instead of simplices, we use hypercubes (or ``cubes'' for short).

\begin{dfn}
A \emph{cubical complex} consists of a finite \emph{vertex set} $V$ and a collection $\Squelta$ of subsets of $V$, called \emph{faces}, satisfying the following conditions:
\begin{itemize}
	\item $\emptyset$ is \emph{not} in $\Squelta$;
	\item for each $v \in V$, $\{v\}$ is in $\Squelta$;
	\item for each $F \in \Squelta$, the set $\hat F \coloneqq \{G \in \Squelta : G \subseteq F\}$ is isomorphic (as a poset ordered by inclusion) to the poset of non-empty faces of a cube; and
	\item if $F$ and $G$ are in $\Squelta$, $F \cap G$ is either empty or in $\Squelta$.
\end{itemize}
\end{dfn}

We usually abuse notation and just write $\Squelta$ (pronounced ``square'') to denote the cubical complex, instead of $(V, \Squelta)$. We will typically use capital Greek letters with horizontal and vertical lines to denote cubical complexes, e.g.\ $\Squelta, \Gamma, \Pi, \Xi, \dotsc$, and other capital Greek letters for simplicial complexes, e.g.\ $\Delta, \Lambda, \dotsc$.

The choice of whether to include or exclude the empty set in the definition of $\Squelta$ is arbitrary, and many authors use the opposite convention, for which the word ``non-empty'' should also be removed from the third condition. We make the choice to exclude the empty set because the set of non-empty faces of a cube is slightly simpler to describe than the set of all faces.

This definition should seem reminiscent of the definition of a simplicial complex. The only differences are the exclusion of the empty set, the fourth condition (which is redundant in the simplicial version), and ``non-empty faces of a cube'' should be replaced by ``faces of a simplex'' in the third condition: the face poset of a simplex is a Boolean lattice, isomorphic to the power set of $\{1, \dotsc, n\}$.

Much of the terminology for simplicial complexes carries over to cubical complexes as well. The \emph{dimension} of a face $F$, denoted $\dim F$, is the dimension of the cube whose face poset $\hat F$ is isomorphic to, and the dimension of a complex is the maximum of the dimensions of its faces. The \emph{$k$-skeleton} of $\Squelta$ is the complex $\Skel_k \Squelta$ whose faces are $\{F \in \Squelta : \dim F \leq k\}$. The $1$-skeleton of $\Squelta$ is also called its \emph{graph}, $\mathcal G(\Squelta)$, and the vertex set is sometimes denoted $V(\Squelta)$.

If $\Pi$ and $\Xi$ are any cubical complexes, define $\Pi \times \Xi$ to be the complex with vertex set $V(\Pi) \times V(\Xi)$ and face set $\{f \times g : f \in \Pi, g \in \Xi\}$.

The prototypical example of a cubical complex is the \emph{unit interval} $I \coloneqq \big\{\{0\}, \{1\}, \{0,1\}\big\}$. This lets us write the poset of non-empty faces of the $n$-dimensional cube as $I^n = I \times \dotsb \times I$. Define $I^0$ to be the complex with a single vertex. We will sometimes use the alternative notation $I \coloneqq \{0, 1, *\}$, where $0 \preceq *$ and $1 \preceq *$ but $0$ and $1$ are incomparable. In this notation, $I^n$ has vertex set $\{0,1\}^n$ and face set $\{0,1,*\}^n$, with the partial order where $(p_1, \dotsc, p_n) \preceq^n (q_1, \dotsc, q_n)$ if and only if $p_i \preceq q_i$ for all $i$. The \emph{boundary} of a cube, denoted $\partial I^n$, is the complex with face set $I^n \setminus \{(*, \dotsc, *)\}$.

A \emph{map of cubical complexes} is a function $f : V(\Pi) \to V(\Xi)$ such that the image of any face of $\Pi$ is a face of $\Xi$. If $f$ is injective, the map is called an \emph{embedding}; if $f$ is bijective and $f^{-1}$ is also a map of cubical complexes, we say $f$ is an \emph{isomorphism}, and we write $\Pi \cong \Xi$. A \emph{subcomplex} of $\Squelta$ consists of a subset $W \subseteq V$ and a subset $\Gamma \subseteq \Squelta$ such that $(W, \Gamma)$ is a cubical complex. In this situation, define $\Squelta \setminus \Gamma$ to be the complex whose face set is $\{F \in \Squelta : {F \cap V(\Gamma) = \emptyset}\}$; in other words, $\Squelta \setminus \Gamma$ is the complex obtained by deleting all vertices of $\Gamma$ from $\Squelta$, along with all faces containing them. We say that $\Gamma$ is a \emph{full} subcomplex of $\Squelta$ (sometimes called an ``induced subcomplex'') if it contains every face of $\Squelta$ involving only vertices in $W$.

Any cubical complex can be given a geometric realisation as a CW complex. If $\Squelta$ is a cubical complex, construct the geometric realisation $\abs{\Squelta}$ by taking the disjoint union of one copy of $[0,1]^k$ for each $k$-dimensional face of $\Squelta$, identifying the subfaces $f \subseteq F$ with subfaces of the associated cube in agreement with the cubical structure, and for each inclusion $f \subseteq F$ of faces, glueing the cubes for $f$ and $F$ along the appropriate face. The geometric realisation of the complex $I^n$ can thus be identified with the usual cube $[0,1]^n \subseteq \mathbb R^n$. Note that $\abs{\Pi \times \Xi}$ is homeomorphic to $\abs{\Pi} \times \abs{\Xi}$. 

Using the CW structure of a cubical complex, we can compute its homology and cohomology groups using cellular homology and cohomology, denoted $H_i(\Squelta)$ and $H^i(\Squelta)$ respectively. Crucially for this paper, the cellular homology groups of a CW complex agree with the singular homology of the complex as a topological space --- that is, $H_i(\Squelta) \cong H_i(\abs{\Squelta})$ --- and the $i$th cellular homology can be computed from just the $j$-skeleton when $i < j$ --- that is, $H_i(\Squelta) \cong H_i(\Skel_j \Squelta)$ --- and similar statements for cohomology \cite[Theorem~2.35, Lemma~2.34(c) and Theorem~3.5]{book:Hatcher}. In this paper, if no coefficients are written, assume that all coefficients of homology are $\mathbb Z/2 \mathbb Z$ (although many of the results can be easily modified for any sensible choices of coefficients), and we only consider non-reduced homology.

We say that a cubical complex $\Mu$ of dimension $d$ is a \emph{cubical homology manifold} if $\abs{\Mu}$ is homeomorphic to a homology manifold (without boundary), that is, for every point $p \in \abs{\Mu}$, the relative homology $H_i(\abs{\Mu}, \abs{\Mu} \setminus p; \mathbb Z)$ is $\mathbb Z$ if $i = d$ and trivial otherwise.

\section{The main result}

The argument in Dancis's paper begins with some preliminary facts about simplicial complexes, then proceeds to a mostly topological proof. The topological reasoning applies to cubical complexes without much modification, but the beginning of the argument needs some changes before it will work in the cubical world.

\subsection{Cubical modifications}

First, Dancis uses the following lemma:
\begin{lma}[{see \citep[lemma 70.1]{book:Munkres}}]
If $\Delta$ is a simplicial complex and $\Lambda$ is a full (sometimes called ``induced'') subcomplex, then $\abs{\Delta} \setminus \abs{\Lambda}$ and $\abs{\Delta \setminus \Lambda}$ are homotopy equivalent.
\end{lma}
To make this lemma work in the cubical world, we need a suitable analogue of fullness.

\begin{dfn}
Let $\Squelta$ be a cubical complex. A subcomplex $\Gamma \subseteq \Squelta$ is \emph{face-like} if for every face $F \in \Squelta$, the intersection $V(\Gamma) \cap F$ is empty or a face of $\Gamma$.
\end{dfn}

For example, for any face $G$ of a cubical complex, the subcomplex $\hat G$ induced by $G$ is a face-like subcomplex (hence the name).

The analogous condition for simplicial complexes is equivalent to fullness. For cubical complexes, face-like-ness implies fullness, but the reverse implication is not true: for example, if $\Squelta$ is a solid square and $\Gamma$ is a pair of diagonally opposite vertices, then $\Gamma$ is a full subcomplex but not face-like, since the intersection of $\Gamma$ with the whole square is not a face.

\begin{prop} \label{thm:cubical-deletion-homotopy}
Let $\Squelta$ be a cubical complex and $\Gamma \subseteq \Squelta$ a face-like subcomplex. Then $\abs{\Squelta} \setminus \abs{\Gamma}$ is homotopy equivalent to $\abs{\Squelta \setminus \Gamma}$.
\end{prop}

\begin{figure}
\caption{An example of the deformation retraction in \cref{thm:cubical-deletion-homotopy}. The subcomplex $\Gamma$ to be deleted is shown with dashed lines in \cref{fig:homotopy-1}.} \centering
\begin{subfigure}{0.25\textwidth}
\caption{} \centering \label{fig:homotopy-1}
\begin{tikzpicture}[x={(1.0cm,0.1cm)}, y={(0cm,1cm)}, z={(0.4cm,-0.6cm)}]
\filldraw [face, thin, opaque, draw=none] (0,1,1) -- (1,1,1) -- (1,2,1) -- (0,2,1) -- cycle;
\filldraw [face, thin, opaque] (0,0,0) -- (1,0,0) -- (1,0,1) -- (0,0,1) -- cycle;
\filldraw [face, thin, opaque] (0,0,0) -- (0,1,0) -- (0,1,1) -- (0,0,1) -- cycle;
\filldraw [face, thin, opaque] (0,1,0) -- (0,2,0) -- (0,2,1) -- (0,1,1) -- cycle;
\filldraw [face, thin, opaque] (0,0,1) -- (1,0,1) -- (1,1,1) -- (0,1,1) -- cycle;
\filldraw [face, thin, opaque] (0,2,0) -- (1,2,0) -- (1,2,1) -- (0,2,1) -- cycle;

\node [vertex, inner sep=1.4pt, fill=white, draw=black] (v1) at (1,1,1) {};
\node [vertex, inner sep=1.4pt, fill=white, draw=black] (v2) at (1,2,1) {};
\draw [very thick, dashed] (v1) -- (v2);
\end{tikzpicture}
\end{subfigure} $\rightarrow$
\begin{subfigure}{0.25\textwidth}
\caption{} \centering
\begin{tikzpicture}[x={(1.0cm,0.1cm)}, y={(0cm,1cm)}, z={(0.4cm,-0.6cm)}]
\filldraw [face, thin, opaque] (1/2,1/2,1/2) -- (1,1/2,1/2) -- (1,1,1/2) -- (1/2,1,1/2) -- cycle;
\filldraw [face, thin, opaque] (1/2,1,1/2) -- (1,1,1/2) -- (1,2,1/2) -- (1/2,2,1/2) -- cycle;
\filldraw [face, thin, opaque] (1/2,1/2,1/2) -- (1,1/2,1/2) -- (1,1/2,1) -- (1/2,1/2,1) -- cycle;
\filldraw [face, thin, opaque] (0,0,0) -- (0,1,0) -- (0,1,1) -- (0,0,1) -- cycle;
\filldraw [face, thin, opaque] (0,1,0) -- (0,2,0) -- (0,2,1) -- (0,1,1) -- cycle;
\filldraw [face, thin, opaque] (0,0,1) -- (1,0,1) -- (1,1/2,1) -- (1/2,1/2,1) -- (1/2,1,1) -- (0,1,1) -- cycle;
\filldraw [face, thin, opaque] (0,1,1) -- (1/2,1,1) -- (1/2,2,1) -- (0,2,1) -- cycle;
\filldraw [face, thin, opaque] (0,2,0) -- (1,2,0) -- (1,2,1/2) -- (1/2,2,1/2) -- (1/2,2,1) -- (0,2,1) -- cycle;
\end{tikzpicture}
\end{subfigure} $\rightarrow$
\begin{subfigure}{0.25\textwidth}
\caption{} \centering
\begin{tikzpicture}[x={(1.0cm,0.1cm)}, y={(0cm,1cm)}, z={(0.4cm,-0.6cm)}]
\filldraw [face, thin, opaque] (0,0,0) -- (1,0,0) -- (1,1,0) -- (0,1,0) -- cycle;
\filldraw [face, thin, opaque] (0,1,0) -- (1,1,0) -- (1,2,0) -- (0,2,0) -- cycle;
\filldraw [face, thin, opaque] (0,0,0) -- (1,0,0) -- (1,0,1) -- (0,0,1) -- cycle;
\filldraw [face, thin, opaque] (0,0,0) -- (0,1,0) -- (0,1,1) -- (0,0,1) -- cycle;
\filldraw [face, thin, opaque] (0,1,0) -- (0,2,0) -- (0,2,1) -- (0,1,1) -- cycle;
\end{tikzpicture}
\end{subfigure}
\end{figure}

\begin{proof}
We will prove the statement by first constructing a strong deformation retraction $H_F$ from $\abs{\hat F} \setminus \abs{\Gamma \cap \hat F}$ to $\abs{\hat F \setminus (\Gamma \cap \hat F)}$ for each face $F$ of $\Squelta$, and then observing that wherever faces $F$ and $F'$ overlap, the deformation retractions $H_F$ and $H_{F'}$ agree.

Let $G = V(\Gamma) \cap F$, so $\hat G = \Gamma \cap \hat F$. Since $\Gamma$ is face-like, $G$ is either empty or a face of $\Gamma$, and thus of $F$. If $G$ is empty or $G = F$, then $\abs{\hat F} \setminus \abs{\hat G} = \abs{\hat F \setminus \hat G}$, so we can take $H_F$ to be the constant homotopy. Otherwise, by exploiting the bountiful symmetry of the cube, we may write $\hat F = I^r$ and assume $\hat G$ is the subcomplex $I^k \times 0^{r-k}$. Then
\begin{align*}
\abs{\hat F} \setminus \abs{\hat G} & = \abs{I^r} \setminus \abs{I^k \times 0^{r-k}} \\
& = \abs{I^k} \times (\abs{I^{r-k}} \setminus \abs{0}), \\
\shortintertext{and}
\abs{\hat F \setminus \hat G} & = \abs{I^r \setminus (I^k \times 0^{r-k})} \\
& = \abs{I^k} \times \abs{I^{r-k} \setminus 0}.
\end{align*}
The space $\abs{I^{r-k}} \setminus \abs{0}$ is a topological cube with a single corner point removed, while $\abs{I^{r-k} \setminus 0}$ is a cube where every face containing that corner is removed, which is the subset of the cube consisting of points where some coordinate is $1$.

Now, define a strong deformation retraction $H_F$ from $\abs{I^k} \times (\abs{I^{r-k}} \setminus \abs{0})$ to $\abs{I^k} \times \abs{I^{r-k} \setminus 0}$ by
\begin{align*}
H_F \big( t, (x, y) \big) & \coloneqq \bigg( x, \, (1-t) y + \frac{t y}{\max\limits_{i = 1, \dotsc, r-k} y_i} \bigg).
\end{align*}
Note that $H_F\big(0, (x,y)\big) = (x,y)$ and $H_F\big(1, (x,y)\big) = \big( x, \frac{y}{\max y_i} \big)$, which is in $\abs{I^{r-k} \setminus 0}$ since some coordinate of $\frac{y}{\max y_i}$ is $1$. We leave it to the reader to check the remaining details that this is a well defined deformation retraction.

Now that we have constructed homotopies on the faces $F$ of $\Squelta$, it remains to argue that they agree where faces overlap. By the definition of a cubical complex, the intersection of any two overlapping faces $F$ and $F'$ is a sub-face $f$ of each. Think of $\hat F$ as $I^r = \{0,1,*\}^r$, so the face $F$ is $(*, \dotsc, *)$, and $G$ is $(*, \dotsc, *, 0, \dotsc, 0)$, with $*$ appearing $k$ times. Thus the set of subfaces of $F$ that meet $\Gamma$ is $\{0,1,*\}^k \times \{0,*\}^{r-k}$.

Suppose $f$ is one such face, so all of the last $r-k$ coordinates of $f$ are $0$ or $*$, and let $\iota$ be the inclusion $\abs{\hat f} \hookrightarrow \abs{\hat F}$. The homotopy $H_F$ restricted to $\abs{\hat f}$ is constantly zero in the coordinates where $f$ is $0$, so $\left. H_F \right\vert_{\abs{\hat f}}$ agrees with $\iota \circ H_f$. On the other hand, if $f$ is a subface of $F$ that does not meet $\Gamma$, at least one of the last $(r-k)$ coordinates of $f$ must be $1$, so
\begin{align*}
(1-t)y + \frac{ty}{\max_i y_i} & = (1-t)y + \frac{ty}{1} \\
& = y.
\end{align*}
Thus $H_f$ and the restriction of $H_F$ are both the constant homotopy on $\abs{\hat f}$, so they also agree.

Therefore, for any faces $F$ and $F'$, $H_F$ and $H_{F'}$ agree on their intersection $F \cap F' = f$. Thus the homotopies glue together to form a deformation retraction from $\abs{\Squelta} \setminus \abs{\Gamma}$ to $\abs{\Squelta \setminus \Gamma}$.
\end{proof}

The second modification we make to Dancis's argument is the following lemma, which explains why we only consider subcomplexes of the cube --- this lemma is not true for arbitrary cubical complexes.

\begin{lma} \label{thm:facelike-iff-not-boundary}
Suppose $\Squelta$ is a subcomplex of $I^n$, and $S^k$ is a subcomplex of $\Squelta$ isomorphic to $\partial I^{k+1}$, with $k \geq 1$. Then $S^k$ is a face-like subcomplex of $\Squelta$ if and only if $S^k$ is not the boundary of a $(k+1)$-dimensional face of $\Squelta$, that is, $S^k$ is not the complex $\partial \hat F$ for any $F \in \Squelta$.
\end{lma}

\begin{proof}
If $S^k$ is the boundary of a face $F \in \Squelta$, then $V(S^k) \cap F = V(S^k)$ is not a face of $S^k$, so $S^k$ is not face-like.

On the other hand, even if $S^k$ is not the boundary of any face of $\Squelta$, \cref{thm:Ehrenborg-Hetyei} implies that it is still the boundary of a face $F$ of $I^n$. Suppose $G \in I^n$ is any face. Since $S^k$ and $\hat F$ only differ in the face $F$, we have either $G \cap S^k = S^k$, if $G$ contains $F$, or $G \cap S^k = G \cap F$, which is a face of $S^k$. Hence if $\Squelta$ does not contain $F$, and thus does not contain any face containing $F$, then $S^k$ is a face-like subcomplex of $\Squelta$.
\end{proof}

\subsection{Topological theorems}

From this point on, Dancis's argument is mostly topological, so it carries over almost unchanged for cubical complexes. We restate Dancis's argument here, starting with the following version of Poincar\'e duality (see \citep[Theorem~70.2]{book:Munkres}). Remember that we assume the coefficients of homology are $\mathbb Z/2 \mathbb Z$ if none are written.

\begin{thm}[Poincar\'e duality] \label{thm:Poincare}
If $\Mu$ is a $d$-dimensional cubical homology manifold and $\Gamma$ is a subcomplex, then $H_{j}(\abs{\Mu}, \abs{\Mu} \setminus \abs{\Gamma}) \cong H^{d-j}(\abs{\Gamma})$. If $\Mu$ is orientable, the same holds with coefficients in $\mathbb Z$.
\end{thm}

\begin{lma}[{cf.\ \citep[Lemma~6(b) and Lemma~8]{art:Dancis}}] \label{thm:Sk-homology}
Let $\Mu \subseteq I^n$ be a $d$-dimensional cubical homology manifold, and $S^k$ a subcomplex of $\Mu$ isomorphic to $\partial I^{k+1}$ with $k \geq 2$.
\begin{itemize}
	\item If $S^k$ is the boundary of a face in $\Mu$, then $H_j(\Mu \setminus S^k) \cong H_j(\Mu)$ for all $j \leq d-2$.
	\item If $S^k$ is not the boundary of any face of $\Mu$, then
	\begin{itemize}
		\item $H_j(\Mu \setminus S^k) \cong H_j(\Mu)$ when $j \leq d-2$ and $j \not \in \{d-k, d-k-1\}$, and
		\item $H_j(\Mu \setminus S^k) \not \cong H_j(\Mu)$ for some $j \in \{d-k, d-k-1\}$.
	\end{itemize}
\end{itemize}
\end{lma}

\begin{proof}
First, suppose $S^k$ is the boundary of a face $F$. Consider the long exact sequence of the pair $(\Mu, \Mu \setminus \hat F)$:
\begin{multline*}
\begin{tikzcd}[column sep=small, ampersand replacement=\&]
\dotsb \rar \& H_{j+1}(\Mu, \Mu \setminus \hat F) \rar \& H_{j}(\Mu \setminus \hat F) \rar \& {}
\end{tikzcd} \\
\begin{tikzcd}[column sep=small, ampersand replacement=\&]
{} \rar \& H_{j}(\Mu) \rar \& H_{j}(\Mu, \Mu \setminus \hat F) \rar \& \dotsb
\end{tikzcd}
\end{multline*}
Since $\hat F$ is a face-like subcomplex, \cref{thm:cubical-deletion-homotopy,thm:Poincare} imply that $H_j(\Mu, \Mu \setminus \hat F) \cong H^{d-j}(\hat F)$, which is $0$ when $j < d$. Therefore, away from $j=d$, the exact sequence breaks up into shorter exact sequences:
\begin{equation*}
\begin{tikzcd}[column sep=small]
0 \rar & H_j(\Mu \setminus \hat F) \rar & H_j(\Mu) \rar & 0
\end{tikzcd}
\end{equation*}
But $\Mu \setminus \hat F$ and $\Mu \setminus S^k$ are the same complex, so $H_j(\Mu \setminus S^k) = H_j(\Mu \setminus \hat F) \cong H_j(\Mu)$ when $j \leq d-2$.

Now, suppose $S^k$ is not the boundary of a face. Consider the pair $(\Mu, \Mu \setminus S^k)$:
\begin{multline*}
\begin{tikzcd}[column sep=small, ampersand replacement=\&]
\dotsb \rar \& H_{j+1}(\Mu, \Mu \setminus S^k) \rar \& H_{j}(\Mu \setminus S^k) \rar \& {}
\end{tikzcd} \\
\begin{tikzcd}[column sep=small, ampersand replacement=\&]
{} \rar \& H_{j}(\Mu) \rar \& H_{j}(\Mu, \Mu \setminus S^k) \rar \& \dotsb
\end{tikzcd}
\end{multline*}
In this case, \cref{thm:facelike-iff-not-boundary} says that $S^k$ is a face-like subcomplex of $\Mu$, so \cref{thm:cubical-deletion-homotopy,thm:Poincare} imply that $H_{j}(\Mu, \Mu \setminus S^k) \cong H^{d-j}(S^k)$, which is $\mathbb Z/2 \mathbb Z$ when $j = d-k$ or $j = d$, and $0$ otherwise. Therefore, away from $j = d-k$ and $j = d$, this exact sequence implies that $H_j(\Mu) \cong H_j(\Mu \setminus S^k)$. However, near $j = d-k$ we get the following exact sequence:
\begin{multline} \label{eq:les-facelike-badpart}
\begin{tikzcd}[column sep=small, ampersand replacement=\&]
0 \rar \& H_{d-k}(\Mu \setminus S^k) \rar \& H_{d-k}(\Mu) \rar \& \mathbb Z/2 \mathbb Z \rar \& {}
\end{tikzcd} \\
\begin{tikzcd}[column sep=small, ampersand replacement=\&]
{} \rar \& H_{d-k-1}(\Mu \setminus S^k) \rar \& H_{d-k-1}(\Mu) \rar \& 0
\end{tikzcd}
\end{multline}
So either $H_{d-k}(\Mu \setminus S^k) \not \cong H_{d-k}(\Mu)$ or $H_{d-k-1}(\Mu \setminus S^k) \not \cong H_{d-k-1}(\Mu)$.
\end{proof}

\begin{crl}[{cf.\ \citep[Main Lemma~10]{art:Dancis}}] \label{thm:face-criterion}
If $\Mu \subseteq I^n$ is a $d$-dimensional cubical homology manifold and $S^k \cong \partial I^{k+1}$ a subcomplex with $k \geq 2$, then $S^k$ is the boundary of a $(k+1)$-dimensional face of $\Mu$ if and only if $H_j(\Mu \setminus S^k) \cong H_j(\Mu)$ for both $j = d-k$ and $j = d-k-1$.
\end{crl}

And since the homology groups $H_{d-k}(\Mu)$, $H_{d-k-1}(\Mu)$, $H_{d-k}(\Mu \setminus S^k)$ and $H_{d-k-1}(\Mu \setminus S^k)$ can be computed from the $k$-skeleton of $\Mu$ when $d-k < k$, we conclude:
\begin{thm}[{cf.\ \citep[Theorem~11]{art:Dancis}}] \label{thm:main}
Any $d$-dimensional cubical homology manifold $\Mu \subseteq I^n$ can be reconstructed from its $\left( \left\lfloor \frac{d}{2} \right\rfloor + 1 \right)$-skeleton.
\end{thm}

\begin{proof}
Each $(k+1)$-dimensional face of $\Mu$ shows up in the $k$-skeleton as a subcomplex isomorphic to $\partial I^{k+1}$, so reconstructing the $(k+1)$-skeleton of $\Mu$ from its $k$-skeleton amounts to deciding which subcomplexes $S^k \cong \partial I^{k+1}$ are actually boundaries of faces. \Cref{thm:face-criterion} answers this question when $k \geq \left\lfloor \frac{d}{2} \right\rfloor + 1$.
\end{proof}

\subsection{Some improvements on \cref{thm:main}}

Under some mild assumptions (for example, if $\Mu$ is a sphere), we can tighten this result to the $\left\lceil \frac{d}{2} \right\rceil$-skeleton (which differs from \cref{thm:main} when $d$ is even).

\begin{lma}[{cf.\ \citep[Lemma~12]{art:Dancis}}] \label{thm:face-criterion-tighter}
Suppose $\Mu \subseteq I^n$ is a $d$-dimensional cubical homology manifold with $d = 2r \geq 4$, and either
\begin{itemize}
	\item $H_r(\Mu; \mathbb Z/2 \mathbb Z) = 0$, or
	\item $\Mu$ is orientable and $H_r(\Mu; \mathbb Z)$ is finite.
\end{itemize}
Then a subcomplex $S^r \cong \partial I^{r+1}$ bounds a face of $\Mu$ if and only if $H_{r-1}({\Mu \setminus S^r}) \cong H_{r-1}(\Mu)$ (with coefficients in $\mathbb Z/2 \mathbb Z$ in the first case or $\mathbb Z$ in the second).
\end{lma}

\begin{proof}
In the first case, setting $k = r$ in \cref{eq:les-facelike-badpart} gives the following sequence:
\begin{equation*}
\begin{tikzcd}[column sep=small]
0 \rar & \mathbb Z/2 \mathbb Z \rar & H_{r-1}(\Mu \setminus S^r; \mathbb Z/2 \mathbb Z) \rar & H_{r-1}(\Mu; \mathbb Z/2 \mathbb Z) \rar & 0
\end{tikzcd}
\end{equation*}

In the second case, Poincar\'e duality holds with coefficients in $\mathbb Z$ for orientable manifolds, so we get this sequence:
\begin{equation*}
\begin{tikzcd}[column sep=small]
H_r(\Mu; \mathbb Z) \rar["\phi"] & \mathbb Z \rar & H_{r-1}(\Mu \setminus S^r; \mathbb Z) \rar & H_{r-1}(\Mu; \mathbb Z) \rar & 0
\end{tikzcd}
\end{equation*}
in which the map $\phi$ must be zero, since there are no other homomorphisms from a finite group to $\mathbb Z$.

In both cases, we conclude that $H_{r-1}(\Mu \setminus S^r) \not \cong H_{r-1}(\Mu)$ (with appropriate coefficients) whenever $S^r$ is not the boundary of a face in $\Mu$. Conversely, if $S^r$ is the boundary of a face, \cref{thm:Sk-homology} still says that $H_{r-1}(\Mu \setminus S^r; \mathbb Z/2 \mathbb Z) \cong H_{r-1}(\Mu; \mathbb Z/2 \mathbb Z)$, and a similar proof works with coefficients in $\mathbb Z$ when $\Mu$ is orientable.
\end{proof}

\begin{crl}[{cf.\ \citep[Theorem~13]{art:Dancis}}] \label{thm:main-tighter}
If $\Mu$ is a $2r$-dimensional cubical homology manifold satisfying the hypotheses of \cref{thm:face-criterion-tighter}, $\Mu$ is determined by its $r$-skeleton.
\end{crl}

\Cref{thm:main,thm:main-tighter} let us reconstruct manifolds when the dimension is known. \Cref{thm:main-dimensions} will allow some ambiguity in the dimension. But first, a lemma:

\begin{lma} \label{thm:invariance-of-domain}
If $\Mu$ is a connected, $d$-dimensional cubical homology manifold with $d \geq 2$ and $S^d$ is a subcomplex homeomorphic to a $d$-dimensional sphere, then $S^d = \Mu$.
\end{lma}

\begin{proof}
The pair $(\Mu, S^d)$ gives the following long exact sequence:
\begin{equation*}
\begin{tikzcd}[column sep=small, row sep=tiny]
H_{d+1}(\Mu, S^d) \arrow[d, phantom, "=" rotate=-90] \rar & H_d(S^d) \arrow[d, phantom, "=" rotate=-90] \rar & H_d(\Mu) \arrow[d, phantom, "=" rotate=-90] \rar & H_d(\Mu, S^d) \arrow[d, phantom, "\cong" rotate=-90] \rar & H_{d-1}(S^d) \arrow[d, phantom, "=" rotate=-90] \\
0 & \mathbb Z/2 \mathbb Z & \mathbb Z/2 \mathbb Z & H^0(\Mu \setminus S^d) & 0
\end{tikzcd}
\end{equation*}
Thus $H^0(\Mu \setminus S^d) = 0$, so $S^d = \Mu$.
\end{proof}

\begin{thm}[{cf.\ \citep[Theorem~14]{art:Dancis}}] \label{thm:main-dimensions}
Let $\Mu$ and $\Nu$ be cubical homology manifolds that are each subcomplexes of a cube. Let $m = \dim \Mu$ and $n = \dim \Nu$, and assume $m \geq n \geq 3$. Suppose $k$ is an integer with
\begin{itemize}
	\item $k \geq \left\lfloor \frac{m}{2} \right\rfloor + 1$; or
	\item $k = \frac{m}{2}$ and $H_k(\Mu; \mathbb Z/2 \mathbb Z) = H_k(\Nu; \mathbb Z/2 \mathbb Z) = 0$; or
	\item $k = \frac{m}{2}$, $\Mu$ and $\Nu$ are orientable, and $H_k(\Mu; \mathbb Z)$ and $H_k(\Nu; \mathbb Z)$ are finite.
\end{itemize}
Then $\Skel_k \Mu \cong \Skel_k \Nu$ implies $\Mu \cong \Nu$.
\end{thm}

(Our proof is a simplification of Dancis's simplicial version, but with slightly different hypotheses.)

\begin{proof}
The goal of this proof is to reduce to the case $k \geq n$. We will start by assuming $k < n$, and argue that the $(k+1)$-skeletons of $\Mu$ and $\Nu$ are isomorphic, so by induction, the $n$-skeletons must be isomorphic.

To begin with, let us assume $k < n$ and consider the case where $k \geq \left\lfloor \frac{m}{2} \right\rfloor + 1$. Under these assumptions,
\begin{gather*}
n-k \leq m-k \leq m - \left\lfloor \frac{m}{2} \right\rfloor - 1 \leq k-1 \leq n-2 \leq m-2.
\end{gather*}
The key facts are $n-k \leq k-1$ and $m-k \leq k-1$, which allow us to compute the $(n-k)$th, $(m-k)$th and lower homology groups from the $k$-skeleton, and $n-k \leq m-2$ and $m-k \leq n-2$, which will let us use \cref{thm:Sk-homology}.

Suppose $S^k \cong \partial I^{k+1}$ is a subcomplex of $\Skel_k \Mu \cong \Skel_k \Nu$. If $S^k$ bounds a face in $\Mu$, then \cref{thm:Sk-homology} and the fact that $n-k \leq m-2$ imply that $H_j(\Mu \setminus S^k) \cong H_j(\Mu)$ when $j \in \{n-k, n-k-1\}$. Since $n-k \leq k-1$, these homology groups depend only on the $k$-skeleton of $\Mu$, thus $H_j(\Mu \setminus S^k) \cong H_j(\Nu \setminus S^k)$ and $H_j(\Mu) \cong H_j(\Nu)$ when $j \in \{n-k, n-k-1\}$. Therefore, $H_j(\Nu \setminus S^k) \cong H_j(\Nu)$ when $j \in \{n-k, n-k-1\}$, so \cref{thm:face-criterion} implies that $S^k$ bounds a face in $\Nu$.

Similarly, if $S^k$ bounds a face in $\Nu$, then $H_j(\Mu \setminus S^k) \cong H_j(\Nu \setminus S^k) \cong H_j(\Nu) \cong H_j(\Mu)$ when $j \in \{m-k, m-k-1\}$; hence $S^k$ bounds a face in $\Mu$.

When $k = m/2$, the argument is similar. If $n = m$, the result follows from \cref{thm:main-tighter}, so we may assume $n < m$. Then:
\begin{gather*}
n-k \leq m-k-1 = k-1 \leq n-2 \leq m-2.
\end{gather*}
If $S^k$ bounds a face in $\Mu$, \cref{thm:Sk-homology} and the inequality $n-k \leq m-2$ imply that $H_j(\Nu \setminus S^k; \mathbb Z/2 \mathbb Z) \cong H_j(\Mu \setminus S^k; \mathbb Z/2 \mathbb Z) \cong H_j(\Mu; \mathbb Z/2 \mathbb Z) \cong H_j(\Nu; \mathbb Z/2 \mathbb Z)$ when $j \in \{n-k, n-k-1\}$, so $S^k$ bounds a face in $\Nu$ by \cref{thm:face-criterion}. On the other hand, if $S^k$ bounds a face in $\Nu$, \cref{thm:Sk-homology} (or a modification with $\mathbb Z$ coefficients) implies $H_{k-1}(\Mu \setminus S^k) \cong H_{k-1}(\Nu \setminus S^k) \cong H_{k-1}(\Nu) \cong H_{k-1}(\Mu)$ with coefficients in either $\mathbb Z/2 \mathbb Z$ or $\mathbb Z$, and we conclude from \cref{thm:face-criterion-tighter} that $S^k$ bounds a face in $\Mu$. 

Thus the $(k+1)$-dimensional faces in $\Mu$ and $\Nu$ are the same, in all cases when $k < n$; that is, the $(k+1)$-skeletons of $\Mu$ and $\Nu$ are isomorphic. Therefore, we can inductively increase $k$ to reduce to the case where $k \geq n$. In this case, we have $\Skel_n \Mu \cong \Skel_n \Nu = \Nu$. However, we claim that the only skeletons of an $m$-dimensional manifold that are themselves manifolds are the $0$-skeleton and the $m$-skeleton.

We may assume that $\Mu$ and $\Skel_n \Mu$ are connected, since the components of $\Skel_n \Mu$ are precisely the skeletons of the components of $\Mu$ when $n \geq 1$. If $\Skel_n \Mu \cong \partial I^{n+1}$, the only possible $(n+1)$-dimensional face of $\Mu$ is the cube whose boundary is $\Skel_n \Mu$. Then $\Skel_{n+1} \Mu$ is either $\partial I^{n+1}$ or $I^{n+1}$, and there are no possible faces of higher dimension; but $I^{n+1}$ is not a manifold (without boundary), so $\Mu \cong \partial I^{n+1} \cong \Skel_n \Mu$. Conversely, if $\Skel_n \Mu \not \cong \partial I^{n+1}$, then \cref{thm:invariance-of-domain} implies that there are no subcomplexes of $\Skel_n \Mu$ isomorphic to $\partial I^{n+1}$, and again $\Skel_n \Mu = \Mu$. Therefore, $\Mu \cong \Nu$.
\end{proof}

Dancis gives an example of the cyclic $(d+1)$-polytopes, whose $\left\lfloor \frac{d-1}{2} \right\rfloor$-skeletons are the same as the skeleton of the boundary of a simplex, to show that the bounds on $k$ in the simplicial version of \cref{thm:main-dimensions} cannot be improved to $\left\lfloor \frac{d-1}{2} \right\rfloor$. The neighbourly cubical $d$-spheres and $(d+1)$-polytopes constructed by \citep{art:BBC,art:Joswig-Ziegler-NCP,art:Joswig-Rorig}, whose $\left\lfloor \frac{d-1}{2} \right\rfloor$-skeletons match the boundary of a cube, serve the same purpose for cubical manifolds. (These complexes embed into a cube, since $C_1$ in \citep[Theorem~3.1]{art:BBC} is itself a cube and fissuring preserves embeddability into a cube.)

This paper relied on the cubical manifolds being embeddable in a cube. The question remains: what skeleton determines an arbitrary $d$-dimensional cubical manifold?

\bibliographystyle{abbrvnat}
{\bibliography{Reconstruction-refs.bib}} 

\begin{thebibliography}{}

\bibitem[\protect\astroncite{Ardila et~al.}{2012}]{art:Ardila}
Ardila, F., Owen, M., and Sullivant, S. (2012).
\newblock Geodesics in {CAT(0)} cubical complexes.
\newblock {\em Adv. in Appl. Math.}, 48(1):142--163.

\bibitem[\protect\astroncite{Babson et~al.}{1997}]{art:BBC}
Babson, E.~K., Billera, L.~J., and Chan, C.~S. (1997).
\newblock Neighborly cubical spheres and a cubical lower bound conjecture.
\newblock {\em Israel J. Math.}, 102:297--315.

\bibitem[\protect\astroncite{Bayer}{2018}]{art:Bayer}
Bayer, M.~M. (2018).
\newblock Graphs, skeleta and reconstruction of polytopes.
\newblock {\em Acta Math. Hungar.}, 155(1):61--73.

\bibitem[\protect\astroncite{Blass and
  Holszty\'{n}ski}{1972}]{art:Blass-Hosztynski}
Blass, J. and Holszty\'{n}ski, W. (1972).
\newblock Cubical polyhedra and homotopy. {III}.
\newblock {\em Atti Accad. Naz. Lincei Rend. Cl. Sci. Fis. Mat. Natur. (8)},
  53:275--279.

\bibitem[\protect\astroncite{Dancis}{1984}]{art:Dancis}
Dancis, J. (1984).
\newblock Triangulated {$n$}-manifolds are determined by their
  {$[n/2]+1$}-skeletons.
\newblock {\em Topology Appl.}, 18(1):17--26.

\bibitem[\protect\astroncite{Dolbilin et~al.}{1994}]{art:Dolbilin-et-al2}
Dolbilin, N.~P., Shtan'ko, M.~A., and Shtogrin, M.~I. (1994).
\newblock Cubic manifolds in lattices.
\newblock {\em Izv. Ross. Akad. Nauk Ser. Mat.}, 58(2):93--107.

\bibitem[\protect\astroncite{Ehrenborg and Hetyei}{1995}]{art:Ehrenborg-Hetyei}
Ehrenborg, R. and Hetyei, G. (1995).
\newblock Generalizations of {B}axter's theorem and cubical homology.
\newblock {\em J. Combin. Theory Ser. A}, 69(2):233--287.

\bibitem[\protect\astroncite{Funar}{1999}]{art:Funar}
Funar, L. (1999).
\newblock Cubulations, immersions, mappability and a problem of {H}abegger.
\newblock {\em Ann. Sci. \'{E}cole Norm. Sup. (4)}, 32(5):681--700.

\bibitem[\protect\astroncite{Hatcher}{2002}]{book:Hatcher}
Hatcher, A. (2002).
\newblock {\em Algebraic topology}.
\newblock Cambridge University Press, Cambridge.

\bibitem[\protect\astroncite{Havel and Mor\'{a}vek}{1972}]{art:Havel-Moravek}
Havel, I. and Mor\'{a}vek, J. (1972).
\newblock {$B$}-valuations of graphs.
\newblock {\em Czechoslovak Math. J.}, 22(97):338--351.

\bibitem[\protect\astroncite{Joswig and R\"{o}rig}{2007}]{art:Joswig-Rorig}
Joswig, M. and R\"{o}rig, T. (2007).
\newblock Neighborly cubical polytopes and spheres.
\newblock {\em Israel J. Math.}, 159:221--242.

\bibitem[\protect\astroncite{Joswig and Ziegler}{2000}]{art:Joswig-Ziegler-NCP}
Joswig, M. and Ziegler, G.~M. (2000).
\newblock Neighborly cubical polytopes.
\newblock {\em Discrete Comput. Geom.}, 24(2-3):325--344.
\newblock The Branko Gr\"{u}nbaum birthday issue.

\bibitem[\protect\astroncite{Kaczynski et~al.}{2004}]{book:Kaczynski-et-al}
Kaczynski, T., Mischaikow, K., and Mrozek, M. (2004).
\newblock {\em Computational homology}, volume 157 of {\em Applied Mathematical
  Sciences}.
\newblock Springer-Verlag, New York.

\bibitem[\protect\astroncite{Munkres}{1984}]{book:Munkres}
Munkres, J.~R. (1984).
\newblock {\em Elements of algebraic topology}.
\newblock Addison-Wesley Publishing Company, Menlo Park, CA.

\bibitem[\protect\astroncite{Ziegler}{1995}]{book:Ziegler}
Ziegler, G.~M. (1995).
\newblock {\em Lectures on polytopes}, volume 152 of {\em Graduate Texts in
  Mathematics}.
\newblock Springer-Verlag, New York.

\end{thebibliography}

\end{document}